\documentclass[12pt]{article}
\usepackage[margin=1.4in]{geometry}
\usepackage{amssymb,amsmath}              
\usepackage{graphicx}
\usepackage{amsthm}
\usepackage{amssymb}
\usepackage{color}
\usepackage{epstopdf}
\renewcommand{\epsilon}{\varepsilon}
\DeclareMathOperator{\dvg}{div} \DeclareMathOperator{\spt}{spt}
\DeclareMathOperator{\dist}{dist} 
\DeclareMathOperator{\graph}{graph}

  \DeclareMathOperator{\proj}{proj}
  \DeclareMathOperator{\area}{Area}

  \DeclareMathOperator{\Id}{Id}
    \DeclareMathOperator{\Lip}{Lip}
    
\def\R{\mathbb{R}}

\def\d{\delta}
\def\a{\alpha}
\def\e{\epsilon}
\def\t{\tau}

\def\r{\rho}

\def\g{\gamma}

\def\Th{\Theta}
\def\b{\beta}

\def\s{\sigma}
\def\o{\omega}
\def\l{\lambda}
\def\k{\kappa}

\def\H{\mathcal{H}}

\def\mass{\underline{\underline{M}}}
\def\var{\underline{\underline{\text{v}}}}

\def\mv{\mu_{{V}}}
\def\dmv{\,d \mu_{{V}}}

\def\ov{\overline}

\def\res{\hbox{ {\vrule height .25cm}{\leaders\hrule\hskip.2cm}}\hskip5.0\mu}

\def\loc{\text{loc}}
\newtheorem{theorem}{Theorem}[section]
\newtheorem{lemma}[theorem]{Lemma}

\newtheorem{remark}[theorem]{Remark}

\newtheorem{definition}[theorem]{Definition}

\newtheorem{claim}[theorem]{Claim}

\newtheoremstyle{TheoremNum}
        {\topsep}{\topsep}              
        {\itshape}                      
        {}                              
        {\bfseries}                     
        {.}                             
        { }                             
        {\thmname{#1}\thmnote{ \bfseries #3}}
    \theoremstyle{TheoremNum}

\title{An Allard type regularity theorem for varifolds with H\"older continuous generalized normal}
\author{Theodora Bourni  \and Alexander Volkmann}
\date{}
\begin{document}
\maketitle

\begin{abstract}
We prove that  Allard's regularity theorem holds for rectifiable $n$-di- mensional varifolds $V$ assuming a weaker condition on the first variation. This, in the special case when $V$ is a smooth manifold translates to the following: If $\o_n^{-1}\r^{-n}\area(V\cap B_\r(x))$ is sufficiently close to 1 and the unit normal of $V$ satisfies a $C^{0,\a}$ estimate, then $V\cap B_{\r/2}(x)$ is the graph of a $C^{1,\a}$ function with estimates. Furthermore, a similar boundary regularity theorem is true.

\end{abstract}

\section{Introduction}\label{intro}

In 1972 Allard \cite{al1} proved a remarkable regularity theorem for rectifiable $n$-varifolds $V=\var(M,\theta)$ in $\R^{n+k}$ (cf. Theorem \ref{allard}). His theorem roughly says that if the generalized mean curvature of $V$ is in  $L^p_\loc(\mu_V)$, $p>n$, if $\theta\ge 1$ $\mv$-a.e. and if $\o_n\r^{-n}\mv( B_\r(x))$  is sufficiently close to 1 then $\spt V\cap B_{\r/2}(x)$ is a graph of a $C^{1,\a}$ function with estimates, where $\a=1- n/p$, see below for precise definitions.

The purpose of this work is to weaken the condition on the generalized mean curvature of $V$ (cf. Theorem \ref{main}). In particular we show that Allard's regularity theorem still holds if instead we assume that $V$ has \emph{generalized normal of class $C^{0,\alpha}$} in the following sense.  
\begin{definition}\label{C0a-normal}
Let $U$ be an open subset of $\R^{n+k}$ and let \em{$V=\var(M,\theta)$} be a rectifiable $n$-varifold in $U$.
We say that $V$ has {\em generalized normal  of class $C^{0,\alpha}$ in $U$ } if there exists a $K\geq 0$ such that for all $B_\r(x) \subset U$ and all $X \in C_c^1(B_\r(x),\R^{n+k})$
\begin{equation*}\label{varcond}
\d V(X)\le K\r^\a\int_M \|d^MX\|\dmv,
\tag{$\star$}\end{equation*}
where $d^MX:=DX\circ P_{TM}$, $P_{T_xM}$ denoting the orthogonal projection matrix of $\R^{n+k}$ onto $T_xM$, the approximate tangent space of $M$ at $x$, and where for a matrix $A=(a_{ij})$, $\|A\|$ is the euclidean operator norm, i.e. $\|A\|=\sup_{|v|=1}|Av|$.
\end{definition}

The paper is organized as follows. Firstly, we fix notation (mainly following the notation of \cite{si1}) and specify the setting we will be working with, and then we give the exact statement of our main theorem (Theorem~\ref{main}).
Afterwards, we motivate condition \eqref{varcond} by showing that it is satisfied by smooth manifolds and is implied by the hypotheses of Allard's regularity theorem.

In section~\ref{monotonicity section} we prove a monotonicity formula and a Poincar\'e inequality for varifolds with generalized normal of class $C^{0,\a}$, which is a fundamental tool in the proof of the main theorem 
(Theorem~\ref{main}). The proof is given in section~\ref{mainsection}. In section~\ref{bdry} we state the boundary regularity analogue (Theorem \ref{main}), which is a consequence of Theorem~\ref{main} and Allard's boundary regularity theorem \cite{al2} (see also \cite{bourni2}).
In section~\ref{GV} we extend the notion of generalized normal of class $C^{0,\alpha}$ to the class of general varifolds, and prove compactness and rectifiability  theorems.
In section~\ref{applications} we apply Theorem~\ref{main} to solutions of the prescribed mean curvature equation and get regularity estimates for graphs of such solutions.

\bigskip
\textbf{Acknowledgements:}
{We would like to thank Ulrich Menne for useful conversations.}

\subsection*{Notation}

Throughout this paper $U$ will be an open subset of $\R^{n+k}$ and $V=\var(M, \theta)$ will denote an n-rectifiable varifold in $U$, so that $M$ is a countably n-rectifiable $\H^n$-measurable subset of $U$ and $\theta$, the {\em multiplicity function}, is a positive and locally $\H^n$ integrable function on $M$. The associated Radon measure will be denoted by $\mv:=\H^n\res\theta$, so that for any $\H^n$-measurable $A\subset \R^{n+k}$ we have
\[\mv(A)=\int_{A\cap M}\theta\,d\H^n.\]
The {\em first variation} of  $V$ with respect to $X\in C^1_c(U,\R^{n+k})$ is given by
\[\d V(X)=\int_M\dvg_MX\dmv.\]
We say that $V$ has {\em generalized mean curvature }$\vec H$ in $U$ if 
\begin{equation}\label{H}
\d V(X)=\int_M\dvg_MX\dmv=-\int_MX\cdot \vec H\dmv\,\,,\,\,\,\forall X\in C^1_c(U),
\end{equation}
where $\vec H$ is a locally $\mv$-integrable function on $M\cap U$ with values in $\R^{n+k}$. 
We remark that using the Riesz representation theorem such an $\vec H$ exists if the total variation $\|\d V\|$ is a Radon measure in $U$ and moreover $\|\d V\|$ is absolutely continuous with respect to $\mv$ (see \cite{si1} for details).

We now have all the necessary language to state our theorem. We will use the following hypotheses
\begin{equation}\label{hyp}
\left.\begin{split} 1\le \theta\,\,\, \mv\text{-a.e. ,  }0\in\spt V\,\,,&\,B_\r(0)\subset U\\
\o_n^{-1}\r^{-n}\mv(B_\r(0))\le& 1+\d.
\end{split}\,\,\,\,\right\}\tag{h}
\end{equation}

\begin{theorem}\label{main} There exist $\d=\d(n,k,\a)$ and $\gamma=\gamma(n,k,\a)$ $\in (0,1)$ such that if  \emph{$V=\var(M,\theta)$}  satisfies hypotheses \eqref{hyp} and has generalized normal of class $C^{0,\a}$ in $U$ in the sense of {\em{Definition} \ref{C0a-normal}}, with $K\r^\a\le \d$, then $\spt V\cap B_{\gamma\r}(0)$ is a graph of a $C^{1, \a}$ function with scaling invariant $C^{1, \a}$ estimates depending only on $n,k,\a,\d$.
\end{theorem}

For convenience we also state Allard's regularity theorem. 
\begin{theorem}[Allard's Regularity Theorem]\label{allard} For $p>n$, there exist $\d=\d(n,k,p)$ and $\gamma=\gamma(n,k,p)$ $\in (0,1)$ such that if \emph{$V=\var(M,\theta)$}  satisfies hypotheses \eqref{hyp}  and has generalized mean curvature $\vec H$ in $U$ (see \eqref{H}) satisfying
\[\left(\int_{B_\r(0)}|\vec H|^p\dmv\right)^\frac1p\r^{1-\frac np}\le \d \]
then $\spt V\cap B_{\gamma\r}(0)$ is a graph of a $C^{1, 1-\frac np}$ function with scaling invariant $C^{1, 1-\frac np}$~estimates depending only on $n,k,p,\d$.
\end{theorem}

\subsection*{The decay condition \eqref{varcond}  when $V$ is a smooth manifold}

  Before getting to the proof of Theorem \ref{main} we want to motivate the decay condition \eqref{varcond} of the first variation by showing that it holds  when $V$ is a smooth manifold. We first do that in the special case when $V$ is actually given by the graph of a smooth function.
  
  \smallskip 
 \noindent \textbf{Smooth graphs}

Let  $M=\graph u\subset\R^{n+1}$ be a graph over $B_\r^n(0) \subset \R^n$ with $u(0)=Du(0)=0$. 
The downward normal to the graph
is given by
 \[\nu=\frac{1}{\sqrt{1+|Du|^2}}\left(D_1u, D_2u, \dots D_nu, -1\right).\]
Now let $X(x,x_{n+1})=v(x) e_{n+1}$, where $v\in C^\infty_0(B_\r^n(0))$.
Then we have
\[\begin{split} 
\dvg_M (X) = - D_jv \nu^j \nu_{n+1} = \frac{ Dv \cdot Du }{ 1+|Du|^2 } 
\end{split}\]
and
\[ \| d^MX \| = \frac{\sqrt{(1+|Du|^2)|Dv|^2-(Du\cdot Dv)^2}}{\sqrt{1+|Du|^2}} \geq  \frac{|Dv|}{\sqrt{1+|Du|^2}},\]
where we also view the functions $v$ and $u$ as functions on $B_\r^n(0) \times \R $, that are independent of the $x_{n+1}$-variable.
Since $\H^n\llcorner M = \mathcal{L}^n \llcorner \sqrt{1+ |Du|^2}$, we conclude that
\[\begin{split}
\delta V(X) & =\int_M  \dvg_MX \,d\H^n =  \int_{B_\r^n(0)}  \frac{Dv  \cdot Du }{ \sqrt{ 1+|Du|^2 } } \,d\mathcal L^n \\
&=  \int_{B_\r^n(0)} \frac{ Dv }{ \sqrt{ 1+|Du|^2 } }  \cdot ( Du - Du(0) ) \,d\mathcal L^n \\
&  \leq \int_M \|d^MX\| \, | Du  - Du(0) | \,d\mathcal H^n\\
& \leq [Du ]_{\alpha, B_\r^n(0)}\, \rho^\alpha \int_M \|d^MX \|  \,d\mathcal H^n.
\end{split}\]

\noindent\textbf{Smooth manifolds}
  
Let $M$ be an $n$-dimensional manifold in $\R^{n+k}$ and let $\{\t_i\}_{i=1}^n$  be a local orthonormal frame of $M$ about $0\in M$. Then for $X\in C^1_c(B_\r(0), \R^{n+k})$ we have  (for $\r$ small enough)

\[\begin{split}\d M(X)&=\int_M\dvg_MX\,d\H^n= \sum_{i=1}^{n}\int_M (D_{\t_i} X) \cdot \t_i\,d\H^n\\
&= \sum_{i=1}^{n}\int_M (D_{\t_i} X) \cdot (\t_i-\t_i(0))\,d\H^n 
	+ \sum_{i=1}^{n}\int_M  (D_{\t_i} X) \cdot \t_i(0)\,d\H^n\\
&= \sum_{i=1}^{n}\int_M  (D_{\t_i} X) \cdot (\t_i-\t_i(0))\,d\H^n\\
&\leq K \r^\alpha \,\int_M  \| d^M X \| \,d\H^n,
\end{split}\]
where $K$ is a constant that depends on a suitably defined local $\a$-H\"older seminorm of the normal $\nu$ to $M$ in $ M \cap B_\r(0)$.

\subsection*{Generalized mean curvature in $L^p$ implies generalized normal of class $C^{0,\alpha}$ }
In this section we show that if $V=\var(M,\theta)$ satisfies conditions \eqref{hyp}, with any $\d$ (not necessarily small) and some $\r \in (0,1]$, and has generalized mean curvature such that 
\[\left(\frac{1}{\omega_n}\int_{B_\r(0)}|\vec H|^p\dmv\right)^\frac1p\r^{1-\frac np}\le \Gamma \left(1-\frac{n}{p} \right)\]
for some $p>n$ and for some $\Gamma \in [0,1/2]$, then $V$ satisfies the decay condition \eqref{varcond} in $B_{\g\r}(0)$ for some $\g = \g(n,k,p,\d) \in (0,1)$ and with $\a = 1-n/p$.

Let $X \in C_c^1(B_{\g\r}(0),\R^{n+k})$. We estimate with H\"older's inequality
\begin{align}\label{H in Lp implies nu in C0alpha}
\delta V(X) 
&\leq \int_M |\vec H||X|\dmv \nonumber\\
&\leq \left( \int_{B_{\g\r}(0)} |\vec H|^p \dmv \right)^\frac{1}{p} \mu_V(B_{\g\r}(0))^{(1-\frac{n}{p})\frac{1}{n}}  \left( \int_M |X|^\frac{n}{n-1} \dmv \right)^\frac{n-1}{n} \nonumber\\
&\leq c(n,p, \d)
  \g^\a  \left( \int_M |X|^\frac{n}{n-1} \dmv \right)^\frac{n-1}{n},
\end{align}
where we have used the monotonicity formula for the area ratios (see \cite[Theorem 17.6, Remark 17.9]{si1}).
Together with the Michael-Simon inequality (see \cite[Theorem 18.6]{si1}) applied to the functions $X^i$, $i=1,...,n+k$, we obtain
\[\begin{split}
\left( \int_M |X|^\frac{n}{n-1}\dmv \right)^\frac{n-1}{n}
& \leq c(n,k) \left( \int_M \|d^MX \| +|X||\vec H| \dmv \right) \\
& \leq c(n,k) \int_M \|d^MX \|\dmv \\
&\quad + c(n,k,p,\d)\gamma^\alpha  \left( \int_M |X|^\frac{n}{n-1} \dmv \right)^\frac{n-1}{n}.
\end{split}\]
 Hence, for $\g \leq \g_0(n,k,p,\d)$ we obtain upon absorbing 
\[ \left( \int_M |X|^\frac{n}{n-1}\dmv \right)^\frac{n-1}{n} \leq c(n,k,p,\d) \int_M \| d^MX \|\dmv . \]
Inserting this into \eqref{H in Lp implies nu in C0alpha} we infer
\[ \delta V(X)\leq c(n,k,p,\d) \g^\a\int_M \| d^MX \|\dmv  ,\] 
which is exactly the decay condition \eqref{varcond} with $K=c(n,k,p,\d)$.

\section{Monotonicity formula}\label{monotonicity section}

In this section we show that a varifold  $V=\var(M,\theta)$, which satisfies the decay condition  \eqref{varcond}, satisfies some nice monotonicity properties, similar to those for varifolds with generalized mean curvature satisfying an $L^p$ estimate, $p>n$, (cf.  \cite[Chapter 4]{si1}).

\begin{lemma}\label{monotonicity}
Assume that \emph{$V=\var(M,\theta)$}  has generalized normal of class $C^{0,\a}$ in $U$ in the sense of {\em{Definition} \ref{C0a-normal}}. Then for any $x\in U$ and all  $0<\s<\r$, with $\r$ such that $\ov B_\r(x)\subset U$ and $K\r^\a\le 1/2$, with $K$ as in condition \eqref{varcond}, we have the following monotonicity formulae.

\begin{itemize}
\item[{\em (i)}]
$e^{K_0\r^\a}\r^{-n}\mv(B_\r(x))\ge e^{K_0\s^\a}\s^{-n}\mv(B_\s(x))+ \frac12\int_{B_\r(x)\setminus B_\s(x)} \frac{|(y-x)^\perp|^2}{r^{n+2}}\dmv$,
\item[{\em (ii)}]
$e^{-K_0\r^\a}\r^{-n}\mv(B_\r(x))\le e^{-K_0\s^\a}\s^{-n}\mv(B_\s(x))+ 2\int_{B_\r(x)\setminus B_\s(x)} \frac{|(y-x)^\perp|^2}{r^{n+2}}\dmv$,
\end{itemize}
where $r=r(y)=|y-x|$, $(y-x)^\perp= P_{N_xM}(y-x)$, with $P_{N_xM}$ denoting the orthogonal projection matrix of $\R^{n+k}$ onto $N_xM=(T_xM)^\perp$,  and $K_0=\frac{n+1}{\a}2K$.

\end{lemma}
\begin{proof}
Without loss of generality we assume that $x=0$ and we write $B_\r=B_\r(0)$. We will use inequality \eqref{varcond} with  the vector field
\[X(x)=\gamma(r) x, \]
where $r=r(x)=|x|$ and $\gamma:\R\to[0,1]$ is a smooth decreasing function such that $\gamma(r)=0$, for $r\ge \r$.
We have
\[ \|d^MX\| \leq \left\| \g(r) \Id + r \g'(r) \left( \frac{x}{r} \otimes \frac{x}{r} \right) \right\|
\leq  \g(r) - r  \g'(r)\]
and thus
\[\begin{split}\int_M\|d^MX\|\dmv&\le  \int_M\gamma(r)\dmv-\int_M r\gamma'(r) \dmv.\end{split}\]
Furthermore,
\[\dvg_M X=n\gamma(r)+r\gamma'(r)\left(1-|D^\perp r|^2\right),\]
where $D^\perp r=\proj_{N_xM}(Dr)$ and $\spt X\subset B_\r$. Hence by plugging the vector field $X$ in \eqref{varcond} we get
\begin{equation*}\begin{split}\int_M n\gamma(r)\dmv+\int_Mr\gamma'(r)\dmv\le 
&\int_Mr\gamma'(r){|D^\perp r|^2}\dmv\\
&+ K\r^{\a}\left(\int_M\gamma(r)\dmv-\int_M r\gamma'(r) \dmv \right).\end{split}
\end{equation*}
We work now as in the $\vec H\in L^p$ case (see \cite[\S 17]{si1}) by setting $\gamma(r)=\phi(r/\rho)$ where $\phi:\R\to [0,1]$ is a smooth function such that $\phi(t)=0$ for $t\ge1$ and $\phi'(t)\le 0$ for all $t$. Since $r\gamma'(r)=r\r^{-1}\phi'(r/\rho)=-\rho\frac{\partial}{\partial\rho}(\phi(r/\rho))$ and after multiplying by $-\rho^{-n-1}$, we get

\[\begin{split}\frac{\partial}{\partial\r}\left(\r^{-n}\int_M\phi(r/\rho)\dmv\right)\ge & \r^{-n}\frac{\partial}{\partial\rho}\int_M\phi(r/\rho) {|D^\perp r|^2}\dmv\\
&-K\r^{\a-n-1}\left(\int_M\phi(r/\r)\dmv+\r\frac{\partial}{\partial\r} \int_M \phi(r/\r) \dmv\right)\\
=& \r^{-n}\frac{\partial}{\partial\rho}\int_M\phi(r/\rho) {|D^\perp r|^2}\dmv \\
&-K\r^\a\left(\frac{\partial}{\partial\r}\left(\r^{-n}\int_M\phi(r/\rho)\dmv\right)\right)\\
&-K\r^{\a-n-1}(1+n)\int_M\phi(r/\r)\dmv\end{split}\]
and thus
\[\begin{split}\frac{\partial}{\partial\r}\left(\r^{-n}\int_M\phi(r/\rho)\dmv\right)+\frac{K\r^{\a-1}(n+1)}{1+K\r^a}&\r^{-n}\int_M\phi(r/\r)\dmv\\
\ge&\frac{1}{1+K\r^\a}\r^{-n}\frac{\partial}{\partial\rho}\int_M\phi(r/\rho)  {|D^\perp r|^2}\dmv.
\end{split}\]

\noindent Finally using $K\r^a\le 1$ and letting $\phi$ increase to the characteristic function of the interval $(-\infty, 1)$   we have in the distributional sense (see \cite[Lemma 14.1]{schaetzle}) 

\begin{equation}\label{monoi}\begin{split}\frac{\partial}{\partial\r}\left(\r^{-n}\mv(B_\r)\right)+\frac{K\r^{a-1}(n+1)}{1+K\r^\a}\r^{-n}\mv(B_\r)\ge &\frac{1}{1+K\r^\a} \frac{\partial}{\partial\rho}\int_{B_\r} \frac{|x^\perp |^2}{r^{n+2}}\dmv.
\end{split}\end{equation}
Similarly by using the vector field $-X$, instead of $X$ and working as above we get

\[\begin{split}\frac{\partial}{\partial\r}\left(\r^{-n}\mv(B_\r)\right)-\frac{K\r^{\a-1}(n+1)}{1-K\r^\a}\r^{-n}\mv(B_\r)\le &\frac{1}{1-K\r^\a} \frac{\partial}{\partial\rho}\int_{B_\r} \frac{|x^\perp |^2}{r^{n+2}}\dmv.
\end{split}\]
Let now 
\[K_0= \frac{n+1}{\a}2K.\]
Then, using the hypothesis $K\r^\a\le 1/2$ we get
\[\begin{split}\frac{\partial}{\partial\r}\left(\r^{-n}\mv(B_\r)\right)+\a K_0\r^{\a-1}\r^{-n}\mv(B_\r)\ge &\frac{1}{2} \frac{\partial}{\partial\rho}\int_{B_\r}\frac{|x^\perp |^2}{r^{n+2}}\dmv
\end{split}\]
and

\[\begin{split}\frac{\partial}{\partial\r}\left(\r^{-n}\mv(B_\r)\right)-\a K_0\r^{\a-1}\r^{-n}\mv(B_\r)\le &2 \frac{\partial}{\partial\rho}\int_{B_\r} \frac{|x^\perp |^2}{r^{n+2}}\dmv
\end{split}\]
and multiplying these inequalities by $e^{K_0\r^\a}$ and $e^{-K_0\r^\a}$ respectively we get 
\[\frac{\partial}{\partial \rho}\left(e^ {K_0\r^\a}  \r^{-n}\mv(B_\r)\right)\ge  \frac{e^ {K_0\r^\a} }{2} \frac{\partial}{\partial\rho}\int_{B_\r} \frac{|x^\perp |^2}{r^{n+2}}\dmv\]
and 
\[\frac{\partial}{\partial \rho}\left(e^ {-K_0\r^\a}  \r^{-n}\mv(B_\r)\right)\le  2{e^ {-K_0\r^\a} } \frac{\partial}{\partial\rho}\int_{B_\r} \frac{|x^\perp |^2}{r^{n+2}}\dmv.\]
Integrating these from $0<\s<\r$ gives the result.

\end{proof}

\subsection*{Poincare-type inequality}
In the previous section we have a monotonicity formula for a quantity involving $\int_{M\cap B_\r}\dmv$. Now we want to extend this to a monotonicity formula for a quantity involving $\int_{M\cap B_\r} h\dmv$, for a positive smooth function $h$.

\begin{lemma} \label{monotonicityh} Assume that \emph{$V=\var(M,\theta)$}  has generalized normal of class $C^{0,\a}$ in $U$ in the sense of {\em{Definition} \ref{C0a-normal}}. Then for any $x\in U$ and all  $0<\s<\r$, with $\r$ such that $\ov B_\r(x)\subset U$ and $K\r^\a\le 1$, with $K$ as in condition \eqref{varcond}, we have the following monotonicity formula for a non negative function $h\in C^1(U)$
\[\begin{split}\frac{1}{\s^n} \int_{B_\s(x)}h\dmv\le& \frac{e^ {K_0\r^\a}}{\r^n} \int_{B_\r(x)}h\dmv -\frac {e^{K_0\r^a}}{2 }\int_{B_\r(x)\setminus B_\s(x) }\frac{|(y-x)^\perp |^2}{r^{n+2}} h\dmv\\
&+\frac{e^ {K_0\r^\a}}{n}  \int_{B_\r(x)}\frac{|\nabla^M h|}{r^{n-1}}\dmv,\end{split}\]
where $K_0= 2 K\frac{n+1}{\a}$ and $r$, $(y-x)^\perp$ are as in Lemma \ref{monotonicity}.

\end{lemma}
\begin{proof}
 Without loss of generality  we assume that $x=0$ and we write $B_\r=B_\r(0)$. We repeat the computations for the monotonicity formula in the proof of Lemma \ref{monotonicity} using now the vector field
\[ X(x)=h(x)\gamma(r) x\]
where $\gamma, r$ are as  in the proof of Lemma \ref{monotonicity}. Since
\[\begin{split} \|d^MX\| \le \g(r) h -r  \g'(r) h+r\gamma(r)|\nabla^M h|\end{split}\]
%
and
\[\dvg_M X=n\gamma(r) h+r\gamma'(r) h\left(1-|D^\perp r|^2\right)+\g(r) x\cdot\nabla^M h\]
by plugging the vector field $X$ in \eqref{varcond} we get as in \eqref{monoi} of Lemma \ref{monotonicity}

\begin{equation}\begin{split}\label{poincaremono}
\frac{\partial}{\partial\r} & \left(\r^{-n}\int_{B_\r} h\,\dmv\right) +\frac{K\r^{\a-1}(n+1)}{1+K\r^\a}\r^{-n} \int_{B_\r} h \dmv\\
&\ge\frac{1}{1+K\r^\a} \frac{\partial}{\partial\r}\int_{B_\r}\frac{|D^\perp r|^2}{r^n} h\dmv- \r^{-n-1} \int_{B_\r}r|\nabla^M h|\dmv .\end{split}
 \end{equation}
Multiplying this by $e^{K_0\r^\a}$, where $K_0=  2K\frac{n+1}{\a}$ and using the hypothesis $K\r^\a\le 1$ we get

\[\begin{split}\frac{\partial}{\partial \rho}\left(e^ {K_0\r^\a}  \r^{-n} \int_{B_\r}h\dmv\right)\ge &\frac {e^{K_0 \r^\a}}{2}  \frac{\partial}{\partial\r}\int_{B_\r}\frac{|x^\perp |^2}{r^{n+2}} h\dmv \\
&-{e^ {K_0\r^\a}}\r^{-n-1} \int_{B_\r} r|\nabla^M h|\dmv. \end{split}\]
Integrating from $\s$ to $\r$ gives the result.
\end{proof}

\section{Proof of Theorem \ref{main}}\label{mainsection}

The proof of Allard's regularity theorem (Theorem \ref{allard}) is based on a Lipschitz approximation and a tilt excess decay theorem, which in turn are derived from the monotonicity formula and the use of special choices of test vector fields in the first variation identity, respectively. For details see \cite{al1} or \cite[Chapter 5; \S 20, \S 22]{si1}.

In our case, where instead of a generalized mean curvature bounded in $L^p$  (as  in Theorem \ref{allard}) the varifold has generalized normal of class $C^{0,\a}$, we show that both a Lipschitz approximation lemma as well as a tilt excess decay theorem are still valid by use of  the monotonicity formulae given in Section~\ref{monotonicity section} and of condition \eqref{varcond}. Having established these two main steps, the proof of Theorem~\ref{main} follows exactly the one of Allard's regularity theorem. For details see \cite{al1} or \cite[Chapter 5; \S 23, \S 24]{si1}.

In what follows we give the exact statements of the Lipschitz approximation and the tilt excess decay theorem for our case and outline their proofs by pointing out the main differences to the corresponding proofs in \cite{si1}.

\subsection*{Lipschitz approximation}

We  define the quantity 
\[ E=R^{-n}\int_{B_R(0)}\|p_{T_xM}-p\|^2\dmv + (KR^\a)^2,\]
where $p_{T_xM}$ and $p$ are the orthogonal projections of $\R^{n+k}$ onto  $T_xM$ and $\R^n$ respectively.
For $V=\var(M,\theta)$ satisfying \eqref{varcond} in $B_R(0)$ we will use the following hypotheses
\begin{equation}\label{hyp'}
\left.\begin{split} 1\le \theta\,\,\, \mv\text{-a.e. ,  }&0\in\spt V\\
\o_n^{-1}R^{-n}\mv(B_R(0))\le& 2(1-a)\text{     for some   }R>0
\end{split}\,\,\,\,\right\}\tag{h'}
\end{equation}

\begin{lemma}\label{LipApprox} There exists a constant $\gamma=\gamma(n,k, \a, a)$ such that if  \emph{$V=\var(M,\theta)$}  has generalized normal of class $C^{0,\a}$ in $B_R(0)$ in the sense of {\em{Definition} \ref{C0a-normal}} and satisfies hypotheses \eqref{hyp'}, then for any $\ell\in (0, 1]$ there exists a Lipschitz function $f= (f^1, \dots, f^k): B^n_{\gamma R}(0)\to \R^k$ with
\[\Lip f\le \ell\,\,\,,\,\,\,\sup |f|\le c E^\frac{1}{2n+2} R\]
and
\[\H^n(B_{\gamma R}(0)\cap (\graph u\setminus M)\cup (M\setminus \graph u))\le c\ell^{-2n-2} E R^n\]
where $c=c(n,k, \a, a)$.
\end{lemma}

\begin{proof} The proof is exactly the same as the one in the case when $V$ has a generalized mean curvature that satisfies an $L^p$ estimate (see \cite{si1}). In particular we show that the set
\[G=\left\{\xi\in M\cap B_{\gamma R}: \r^{-n}\int_{B_\r(\xi)}\|p_{T_xM}-p\|^2\dmv\le \d\ell^{2n+2}\,\,,\,\,\,\forall\r\in(0,  R/10)\right\}\]
for $\gamma=\gamma(n,k, \a, a)$ and $\d=\d(n,k, \a, a)$ small enough, is a Lipschitz graph with lipschitz constant $\ell$ and that $G$ is ``most'' of $M$ in the sense that is required by the lemma. To show this last statement we use the monotonicity  formulae of Section~\ref{monotonicity section} and to show that $G$ is a lipschitz graph we use the following claim (cf. \cite[Lemma 12.5]{si1})

\begin{claim}{\cite[Lemma 12.5]{si1}}\label{techlemma} Let  $\beta\in (0,1)$, $\ell>0$. Suppose $y,z\in B_{\beta R}(0)$ with $|y-z|\ge \b R/4$, $\Th(y), \Th(z)\ge 1$ and $|q(y-z)|\ge \ell |y-z|$, where $q$ is the orthogonal projection of $\R^{n+k}$ onto $\R^k$. Then

\[\begin{split}\Th(y)+\Th(z) \le &\frac{1+ 5K_0R^\a}{(1-\b)^n\o_n}\left(1+c(\ell \beta)^{-n}KR^{1+\a}\right) R^{-n}\mu_V(B_R)\\
&+ (1+ 5K_0R^\a)c(n,k)(\ell \b )^{-n-1}R^{-n}  \int_{B_R}\|p_{T_xM}-p\|\dmv\end{split}\]
where $c$ is an absolute constant and $c(n,k)$ is a constant that depends on $n$ and $k$.
\end{claim} 

We remark that in \cite[Lemma 12.5]{si1} a bound on the mean curvature is assumed but actually for the proof only the monotonicity formula is needed, which we have here as in Section~\ref{monotonicity section}.

\end{proof}

\subsection*{The tilt-txcess decay lemma}
We define the \emph{tilt-excess} $E(\xi,\rho,T)$ (relative to the rectifiable $n$-varifold $V=\var(M,\theta)$) by
\[ E(\xi,\rho,T): = \frac{1}{2} \rho^{-n} \int_{B_\rho(\xi)} |p_{T_xM} - p_T|^2\dmv , \]
whenever $\rho > 0$, $\xi \in \R^{n+k}$ and $T$ is an $n$-dimensional subspace of $\R^{n+k}$. Here
\[|p_{T_xM} - p_T|^2: = {\rm tr}((P_{T_xM} - P_T)^2).\]

\begin{lemma}[Tilt-excess and hight lemma]\label{TEHL}
Suppose that \emph{$V=\var(M,\theta)$}  has generalized normal of class $C^{0,\a}$ in $U$ in the sense of {\em{Definition} \ref{C0a-normal}}. Then for any $\ov B_\r(\xi) \subset U$ and any $n$-dimensional subspace $T \subset \R^{n+k}$ we have
\[\begin{split} 
E(\xi,\rho/2,T)  &\leq c(1+ K\rho^\alpha)  \rho^{-n} \int_{B_ \rho(\xi) }  \frac{ \dist(x-\xi,T)^2}{\rho^2} \dmv  \\
&\quad + c(n,k)(K\rho^\alpha)^2 \rho^{-n}  \mu_V(B_ \rho(\xi)),
\end{split}\] 
where $c$ is an absolute constant and $c(n,k)$ is a constant that depends on $n$ and $k$.

\end{lemma}
\begin{proof}
The proof is the same as in the case when $V$ has a generalized mean curvature $\vec H\in L^p$ (see \cite[Lemma 22.2]{si1}) with the difference that here, after assuming w.l.o.g. that $T=\R^n$ and $\xi=0$, we use the vector field 
\[ X(x) = \zeta^2(x)x',\quad x'=(0,\dots,0,x^{n+1},\dots,x^{n+k})\]
in the estimate \eqref{varcond} instead of using it  the first variation formula. $\zeta $ here is a cut-off function such that $\zeta \equiv 1$ in $B_{\rho/2}(0)$, $\zeta \equiv 0$ outside $B_\rho(0)$, and $|D\zeta | \leq 3/\rho$. In order to estimate the right hand side of \eqref{varcond} with this vector field inserted, we use 

\[\begin{split}\|d^MX\|\le& 2|\zeta||\nabla^M\zeta||x'|+c(n,k)\zeta^2  |p_{T_xM} - p_T|.\end{split}\]

\end{proof}

In order to state the  Tilt-excess Decay Theorem in a convenient manner, we let $\varepsilon, a  \in (0,1)$, $\rho >0 $, and $T$ an $n$-dimensional subspace of $\R^{n+k}$, be fixed, and we shall consider the hypotheses
\begin{equation}\label{tilt-excess hypotheses}
\begin{cases}
1 \leq \theta \leq 1+ \varepsilon\quad \text{$\mu_V$-a.e. in $U$} \\
\xi \in \spt(\mu_V) ,\; B_\rho(\xi) \subset U,\; \frac{ \mu_V(B_\rho(\xi)) }{ \omega_n \rho^n} \leq 2(1- a ) ,\\
E_*(\xi,\rho,T) \leq \varepsilon,
\end{cases} 
\end{equation}
where
\[ E_*(\xi,\rho,T) := \max\left\{ E(\xi,\rho,T), \varepsilon^{-1}(K\rho^\alpha)^2 \right\}.\]

\begin{theorem}[Tilt-excess decay]\label{TED thm}
For any $a\in (0,1)$, $\alpha \in (0,1)$, there are constants $\eta, \varepsilon \in (0,1/2)$, depending only on $n,k,a,\alpha$, such that if  \emph{$V=\var(M,\theta)$}  has generalized normal of class $C^{0,\a}$ in $U$ in the sense of {\em{Definition} \ref{C0a-normal}} and hypotheses \eqref{tilt-excess hypotheses} hold,
then
\begin{equation}\label{excess-decay}
E_*(\xi,\eta\rho,S) \leq \eta^{2\a} E_*(\xi,\rho,T)
\end{equation}
for some $n$-dimensional subspace $S$ of $\R^{n+k}$.
\end{theorem}
\begin{proof}

Let $f$  be the approximating Lipschitz function as in Lemma \ref{LipApprox}.  We show that each component  $f^j$, $j=1,\dots, k$,  of  $f$ is well-approximated by a harmonic function. To do that we use condition \eqref{varcond},
with $X= \zeta e_{n+j}$ ($j =1,...,k$), for some $\zeta \in C_c^1(U)$, which gives, in view of $e_{n+j} = Dx^{n+j}$,
\[\int_M \nabla^Mx^{n+j}  \cdot \nabla^M \zeta \dmv   \leq  \,K\rho^\alpha \int_M \| d^M(\zeta e_{n+j}) \| \dmv= K\rho^\alpha \int_M | \nabla ^M\zeta | \dmv.\]
Using Lemma \ref{LipApprox} and the area estimate of \eqref{tilt-excess hypotheses}, along with the above inequality we obtain
\begin{align}\label{equation 10 TED}
 \rho^{-n} \int_{M_1} \nabla^M\widetilde f^j \cdot \nabla^M \zeta  \dmv  \leq c \| \nabla ^M\zeta \|_{C^0} E_*,
 \end{align}
where $M_1 = M \cap {\rm graph}(f)$ and where $\widetilde f^j$ is defined on $\R^{n+k}$ by $\widetilde f^j(x^1,...,x^{n+k}):= f^j(x^1,...,x^n)$ for $x=(x^1,...,x^{n+k}) \in \R^{n+k}$.

The remaining part of the proof follows the arguments in \cite[Theorem 22.5]{si1}, but uses the above version of the tilt-excess and hight lemma (Lemma \ref{TEHL}), instead of \cite[Lemma 22.2]{si1}.
\end{proof}

\section{Boundary regularity}\label{bdry}

Combining Theorem~\ref{main} with Allard's boundary regularity theorem \cite{al2} (see \cite{bourni2} for $C^{1,\a}$ boundaries) we get the following boundary regularity theorem. We assume that $B$ is an (n-1)-dimensional  $C^{1,\a}$ manifold in $\R^{n+k}$ and assume now that  $V=\var(M,\theta)$ is a rectifiable $n$-varifold in $\R^{n+k}$, 
that has  generalized normal  of class $C^{0,\alpha}$ in $U\setminus B$, where $U$ is an open subset of $\R^{n+k}$. I.e. $V$ satisfies condition \eqref{varcond} of Definition~\ref{C0a-normal} for all $B_\r(x) \subset U$ and all $X \in C_c^1(B_\r(x),\R^{n+k})$ with $X=0$ on $B$.

\noindent We will also use the following hypotheses
\begin{equation}\label{hyp1}
\left.\begin{split} 1\le \theta\,\,\, \mv\text{-a.e. ,  }0\in\spt V\cap B\,\,,&\,B_\r(0)\subset U\\
\o_n^{-1}\r^{-n}\mv(B_\r(0))\le& 1+\frac{\d}{2}\\
\k\r^\a\le & \d,
\end{split}\,\,\,\,\right\}\tag{$h_b$}
\end{equation}
where $\k$ is such that
\[\|p_{Tx B}-p_{T_yB}\|\le \k |x-y|^\a\,\,,\,\,\forall x, y \in B.\]

\begin{theorem}\label{mainb} There are $\d=\d(n,k,\a)$ and $\gamma=\gamma(n,k,\a)$ $\in (0,1)$ such that if \emph{$V=\var(M,\theta)$}  satisfies hypotheses \eqref{hyp1}  and $V$ has generalized normal of class $C^{0,\a}$ in $U\setminus B$ in the sense of {\em{Definition} \ref{C0a-normal}}
with $K\r^\a\le \d$, then $V\cap B_{\gamma\r}(0)$ is a graph of a $C^{1, \a}$ function with scaling invariant $C^{1, \a}$ estimates depending only on $n,k,\a,\d$.
\end{theorem}

\section{General varifolds}\label{GV}
We show that the above monotonicity and regularity results apply to general varifolds that satisfy a condition similar to \eqref{varcond}.

We consider $V$ a (general) $n$-varifold on $U\subset \R^{n+k}$; that is, a Radon measure on $G_n(U)=U\times G(n+k, n)$. For $V$ we have an associated Radon measure $\mu_V$ on $U$ defined by
\[\mu_V(A)=V(\pi^{-1}(A))\,\,,\,\,\,A\subset U\,\,\,,\,\,\,(\pi:(x,S)\mapsto x).\]
The mass $\mass(V)$ of $V$ is defined by 
\[\mass(V)=\mv(U).\]

\begin{definition}\label{C0a-normalV}
Let $U$ be an open subset of $\R^{n+k}$ and let \em{$V$} be an $n$-varifold in $U$.
We say that $V$ has {\em generalized normal  of class $C^{0,\alpha}$ in $U$ } if there exists a $K\geq 0$ such that for all $B_\r(x) \subset U$ and all $X \in C_c^1(B_\r(x),\R^{n+k})$
\begin{equation}\label{varcondV}
\d V(X)\le K\r^\a\int_{G_n(U)} \|DX(y)\circ P_S\| dV(y,S),
\tag{$\star\star$}\end{equation}
where $P_{S}$ denotes the orthogonal projection matrix of $\R^{n+k}$ onto $S$, and where for a matrix $A=(a_{ij})$, $\|A\|$ is the euclidean operator norm, i.e. $\|A\|=\sup_{|v|=1}|Av|$.
\end{definition}

Exactly as in the proof of Lemma \ref{monotonicity} we have the following monotonicity formulae for general varifolds.

\begin{lemma}\label{monotonicityV}
Assume that the n-varifold \emph{$V$} has generalized normal  of class $C^{0,\alpha}$ in $U$ in the sense of \emph{Definition~\ref{C0a-normalV}}. Then for any $x\in U$ and all  $0<\s<\r$, with $\r$ such that $\overline B_\r(x)\subset U$ and $K\r^\a\le 1/2$ we have the following monotonicity formulae

\begin{itemize}
\item[{\em (i)}]
$e^{K_0\r^\a}\r^{-n}\mv(B_\r(x))\ge e^{K_0\s^\a}\s^{-n}\mv(B_\s(x))+ \frac12Q_{\s,\r}(x)$,
\item[{\em (ii)}]
$e^{-K_0\r^\a}\r^{-n}\mv(B_\r(x))\le e^{-K_0\s^\a}\s^{-n}\mv(B_\s(x))+ 2Q_{\s,\r}(x)$,
\end{itemize}
where $K_0=\frac{n+1}{\a}2K$, and where
\[ Q_{\s,\r}(x):= \int_{G_n(B_\r(x)\setminus B_\s(x))} \frac{|P_{S^\perp}(y-x)|^2}{r(y)^{n+2}} dV(y,S),\]
with $r(y)=|y-x|$.

\end{lemma}

A direct consequence of Lemma~\ref{monotonicityV} is the following lemma.
\begin{lemma}\label{semicontinuousTheta}
Assume that the n-varifold \emph{$V$} has generalized normal of class $C^{0,\alpha}$ in $U$ in the sense of \emph{Definition~\ref{C0a-normalV}}. Then the density function
\[ x\mapsto \Theta^n(\mv, x) : =\lim_{\r\to 0} \r^{-n}\mv(B_\r(x))\]
is well defined and is upper semi continuous in $U$.
\end{lemma}
%
%
%

In the class of varifolds satisfying condition \eqref{varcondV} we get similar properties for varifold limits as in the case of varifolds having locally bounded first variation cf. \cite[Theorem 40.6]{si1}. In particular, we have the following result.

\begin{theorem} \label{semicont}Suppose $V_i\to V$ (as Radon measures in $G_n(U)$) and $\Theta^n(V_i, y)\ge 1$ for $\mu_{V_i}$-a.e. $y \in U$, and suppose that each $V_i$ has generalized normal of class $C^{0,\alpha}$ in $U$ in the sense of \emph{Definition~\ref{C0a-normalV}}, satisfying condition \eqref{varcondV} with $K=K_i$, and such that 
\[\sup_{i} K_i<\infty.\]
Then $V$ also has generalized normal of class $C^{0,\alpha}$ in $U$ with $K$ such that 
\[K= \liminf_{i\to\infty} K_i<\infty\]
and furthermore $\Theta^n(\mv,y)\ge 1$ for $\mv$-a.e. $y \in U$.
\end{theorem}

\begin{proof}
Let $B_\r(x)\subset U$ and $X \in C_c^1(B_\r(x),\R^{n+k})$. Then, since $V_i\to V$ we have
\[\begin{split}\d V(X)=&\int_{G_n(U)}\dvg_S X(y)\,dV(y, S)=\lim_{i\to\infty}\int_{G_n(U)}\dvg_S X(y)\,dV_i(y, S)
\\
=&\lim_{i\to\infty}\d V_i(X)\le \liminf_{i\to\infty}\left( K_i\r^\a\int_{G_n(U)} \|DX(y)\circ P_S\| dV_i(y,S) \right).\end{split}\]
Since also
\[\lim_{i\to\infty}\int_{G_n(U)} \|DX(y)\circ P_S\| dV_i(y,S)=\int_{G_n(U)} \|DX(y)\circ P_S\| dV(y,S),\]
we get
\[\d V(X)\le K\r^\a \int_{G_n(U)} \|DX(y)\circ P_S\| dV(y,S),\]
where $K = \liminf_{i\to \infty}K_i$.

To prove the density estimate we note that by Lemma \ref{monotonicityV} (i), applied to each $V_i$, with $\s\to 0$ we get
\[e^{C\r^\a}\o_n^{-1}\r^{-n}\mu_{V_i}(B_\r(x))\ge 1\]
for $\mu_{V_i}$-a.e. $x\in U$ and $\overline B_\rho(x) \subset U$, where $C=\frac{n+1}{\a} 2\sup_i K_i$. Hence, for $\mu_{V}$-a.e. $x\in U$ and a.e. $\r>0$ with $ B_\r(x) \subset U$ 
\[\o_n^{-1}\r^{-n}\mu_{V}(B_\r(x))=\lim_{i \to \infty} \o_n^{-1}\r^{-n}\mu_{V_i}(B_\r(x))\ge e^{-C\r^\a},\]
and by approximation from below
\[\o_n^{-1}\r^{-n}\mu_{V}(B_\r(x))\ge e^{-C\r^\a}\]
for every sufficiently small $\r>0$. Taking $\r\to 0$ we get the required estimate.
\end{proof}

Let $V$ be an $n$-varifold satisfying \eqref{varcondV} with some $K$ and let $x\in U$ be such that $\Theta^n(\mv,x)=\theta_0\in (0,\infty)$. Then, for a sequence $\l_i\downarrow 0$ the rescaled varifolds
\[V_i:=\eta_{x, \l_i\sharp}V\]
also satisfy \eqref{varcondV} with $K= K\l_i^\a$.
Furthermore,
\[\mu_{V_i}(W)=\l_i^{-n}\mu_V(\l_iW),\quad\text{for every set $W$ such that $\lambda_iW\subset \subset U$}.\]
Thus, by the monotonicity formula for $V$ and by compactness for Radon measures, we have that (after passing to a subsequence) $V_i$ converge to a varifold $C$, which is  stationary by Theorem \ref{semicont}. Now we can use the standard monotonicity formula for $C$ to infer that
\[\frac{\mu_C(B_\r(x))}{\o_n\r^n}=\theta_0\,\,,\,\,\,\forall \r>0.\]
More generally,
\[\l^{-n}\mu_C(\eta_{0,\l}(A))=\mu_C(A),\,\forall A\subset \R^{n+k},\l>0.\]
In case $\Theta^n(\mu_C, x)>0$ for $\mu_C$-a.e. $x$, e.g. when
\[\lim_{\r\downarrow 0}\r^{-n}\mv(\{ y\in B_\r(x):\Theta^n(\mv, y)<1\})=0,\]
we also have that
\[\eta_{0,\l\sharp} C=C,\]
i.e. $C$ is a cone.

The following rectifiability and compactness theorems are the analogues of \cite[Theorems 42.2, 42.7]{si1}.

\begin{theorem}[Rectifiability]\label{rectifiability}
Let $V$ be an $n$-varifold  with generalized normal  of class $C^{0,\alpha}$ in $U$ in the sense of \emph{Definition~\ref{C0a-normalV}}, and such that $\Theta^n(\mv,x)>0$ for $\mv$-a.e. $x\in U$. Then $V$ is an $n$-rectifiable varifold (i.e. \emph{$V=\var(M,\theta)$}).
\end{theorem}

{\it Remark on the proof.} In the case when $V$ has locally bounded first variation one shows that $V$ has an approximate tangent space at each point $x$ where
\[\| \delta V\|(B_\r(x)) \leq \Lambda(x) \mu_V(B_\r(x)),\quad\text{for all $B_\r(x) \subset U$.}\]
This condition is used only to show that the monotonicity formula holds at the point $x$. In our case the monotonicity formula holds at every point $x \in U$, and thus the same proof goes through.

\begin{theorem}[Compactness]\label{Compactness}
Let $\{V_i\}$ be a sequence of rectifiable $n$-varifolds with generalized normal of class $C^{0,\alpha}$ in $U$ in the sense of \emph{Definition~\ref{C0a-normalV}} with constants $K_i\ge 0$ and such that
\[\sup_i\{\mu_{V_i}(W)\}<\infty \,\,,\,\,\,\forall W\subset\subset U\text{   ,    }\sup_i K_i<\infty,\]
and $\Theta^n(\mu_{V_i}, x)\ge 1$ for $\mu_{V_i}$-a.e. $x \in U$. 

Then, there exist a subsequence $\{V_{i'}\}$ and a rectifiable $n$-varifold $V$  that has generalized normal of class $C^{0,\alpha}$ in $U$ in the sense of \emph{Definition~\ref{C0a-normal}} with $K= \liminf_{i\to \infty} K_i$, such that $V_{i'} \to V$ (in the sense of Radon measure on $G_n(U)$).
The density satisfies $\Theta^n(\mv, x)\ge 1$ for $\mv$-a.e. $x\in U$.

Moreover, when the $V_i$'s are integer multiplicity then so is $V$.
\end{theorem}

{\it Remark on the proof.} The fact that a subsequence converges and that the limit $V$ that has the properties stated in the theorem follows from the compactness theorem for Radon measures, Theorem \ref{semicont} and Theorem \ref{rectifiability}. To show that $V$ is integer multiplicity when the $V_i$'s are, we note the following.

\[\eta_{\xi, \l\sharp}V\stackrel{\l\to0}{\longrightarrow} \theta_0\var(P), \quad\text{for $\mv$-a.e. $\xi\in\spt \mv$, where $P=T_\xi M, \theta_0=\theta(\xi)$.}\]
Since also 
\[V_i \stackrel{i\to \infty}{\longrightarrow} V\implies \eta_{\xi, \l\sharp}V_i \stackrel{i\to\infty}{\longrightarrow} \eta_{\xi, \l\sharp}V\]
we can get a sequence $W_i:=\eta_{\xi, \l_i\sharp}V_i$, with $\l_i\downarrow 0$, so that
\[W_i\to  \theta_0\var(P).\]
Note also that $W_i$ satisfy \eqref{varcond} with $\l_i^\a K_i\le \l_i^\a\sup_i|K_i|\stackrel{i\to\infty}{\longrightarrow }0$. One then proceeds as in \cite[Remark 42.8]{si1} to show that $\theta_0$ needs to be an integer.

\section{Application to the prescribed mean curvature equation}\label{applications}

In this section we use Theorems~\ref{main} and \ref{mainb} to show that the graph of a function satisfying the prescribed mean curvature equation over a $C^{1,\a}$ domain and with $C^{1,\a}$ prescribed boundary values is a $C^{1,\a}$ manifold with boundary provided that the prescribed mean curvature is given as the divergence of a $C^{0,\a}$ vector field. More precisely, we consider the following situation.

Let $\Omega$ be an open $C^{1,\a}$ domain in $\R^n$, $H\in L_{\loc}^{p}(\ov\Omega\times\R)$, where $p>n+1$, and let $f=(f^1,f^2,\dots, f^n)\in C_{\loc}^{0,\a}(\ov\Omega\times\R;\R^n)$. We consider a weak solution $u\in C_\loc^{1,\a}(\Omega)\cap C_\loc^0(\ov\Omega)$ of the following Dirichlet problem
\[\text{(DP)}\left\{\,\,\,\begin{split}\sum_{i=1}^n D_i\left( \frac{D_i u}{\sqrt{1+|Du|^2}}\right)=&H(x, u(x))+\sum_{i=1}^nD_i(f^i(x, u(x)))\text{    in    }\Omega\\
u=&\phi \text{    on    }\partial\Omega,\end{split}\right.\]
where $\phi\in C^{0,\a}_\loc(\partial\Omega)$.
\begin{theorem}\label{DPthm} Let $u\in C_\loc^{1,\a}(\Omega)\cap C_\loc^0(\ov\Omega)$ be a weak solution of the Dirichlet problem (DP), with $\Omega, H, f, \phi$ as above. Then for any $\eta>0$, there exists $\r_0=\r_0(\eta)$ such that the following holds: for any $\r\le \r_0$ and any $x\in \graph u$ there exist a linear isometry $q$ of $\R^{n+1}$ and a function $\psi_x \in C^{1,\a}( U\cap B_\r(x))$, for a $C^{1,\a}$ domain $U$ of $\R^n$, such that
\[ \graph u\cap B_\r(x)=q(\graph \psi_x)\cap B_\r(x),\]
and
\[\r^{-1}\|\psi_x\|_{C^0(U\cap B_\r(x))}+\|D\psi_x\|_{C^0(U\cap B_\r(x))}+\r^{\a}[D\psi_x]_{\a, U\cap B_\r(x)}\le \eta.\]
\end{theorem}
\begin{remark} The radius $\r_0$ given in Theorem~\ref{DPthm} of course also depends on $\partial\Omega, H, f$ and $\phi$, but the dependence is on their corresponding $C_\loc^{1,\a}$, $L_\loc^p$, $C_\loc^{0,\a}$ and $C_\loc^{1,\a}$ norms. Furthermore this is a local estimate, that means that if the above norms are bounded only in $\Omega'\times \R$, where $\Omega'\subset \Omega$, then the Theorem still holds for all $x\in (\Omega'\times\R)\cap \graph u$.
\end{remark}

\begin{proof}[Proof of Theorem \ref{DPthm}]

Let $M=\graph u$, then $M$ satisfies an almost minimizing property, in particular, there exists $\r_0$ depending only on the $C^{0,\a}$ norm of $f$ such that for any $x\in \ov\Omega\times\R$ and any $\r\le \r_0$ the following holds:
\[\area(M\cap B_\r(x))\le \area (N\cap B_\r(x))+ c\o_n\r^n(\r^{1-\frac{n+1}{p}}\|H\|_{L^p(B_\r(x))}+\r^\a[f]_{\a, B_\r(x)})\]
for any integral $n$-current $N$ with $\partial N=[\![\graph\phi]\!]$ and $\spt (N- [\![M]\!])$ a compact subset of $B_\r(x)\cap (\ov \Omega\times\R)$, and where $c$ in the above inequality is an absolute constant
(cf. \cite[Lemma 2.9, 2.10]{bourni}). Here we used the notation $[\![\graph\phi]\!], [\![M]\!]$ to refer to the $n$-current corresponding to the manifolds $\graph\phi$ and $M$, respectively.

This almost minimizing property implies that for any $\e>0$ there exists a $\r_1=\r_1(\e)>0$ such that for all $\r\le\r_1$
\[\begin{split}
|M\cap B_\r(x)|\le& \o_n\r^n(1+\e)\text{   for all  }x\in \graph u\text{  such that  }\dist(x,\graph\phi)>\r \\
|M\cap B_\r(x)|\le& \o_n\r^n\left(\frac12+\e\right)\text{   for all  }x\in \graph\phi.
\end{split}\]
For the above estimate at points far away from $\partial \Omega \times \R$ see for example \cite{Massari-Miranda:1984}. And for the estimates at the remaining points see \cite[Theorem 3.12, Lemma 2.12]{bourni}.

Applying now Theorems \ref{main} and \ref{mainb} for the interior and the boundary respectively gives the result.

\end{proof}

\section{Final remarks}
One may also consider the following more general situation than condition \eqref{varcond}.
\begin{definition}\label{modulus}
Let $U$ be an open subset of $\R^{n+k}$ and let \em{$V=\var(M,\theta)$} be a rectifiable $n$-varifold in $U$.
We say that $V$ has {\em generalized normal with modulus of continuity $\o$ in $U$ } for a nondecreasing function $\o: (0,\infty) \to (0,\infty)$ with $\lim_{\\r \to 0} \o(\r)=0$ if for all $B_\r(x) \subset U$ and all $X \in C_c^1(B_\r(x),\R^{n+k})$
\[
\d V(X)\le \o(\r) \int_M \|d^MX\|\dmv.
\]

It would be interesting to consider the regularity properties of these varifolds $V=\var(M,\theta)$ under hypotheses \eqref{hyp}. In view of the corresponding results for quasi-minimizers of perimeter (cf. \cite{Ambrosio-Paolini:1999}) it is reasonable to expect local $\alpha$-H\"older continuity of $\spt V$ for all $\alpha<1$.
\end{definition}

\bibliographystyle{plain}

\bibliography{bill}

\end{document}